\documentclass[12pt,twoside,reqno]{amsart}
\usepackage[colorlinks=false,citecolor=blue]{hyperref}
\usepackage{mathptmx, amsmath, amssymb, amsfonts, amsthm, mathptmx, enumerate, color,mathrsfs}
\usepackage{graphicx}
\newcommand{\vertiii}[1]{{\left\vert\kern-0.25ex\left\vert\kern-0.25ex\left\vert #1 
    \right\vert\kern-0.25ex\right\vert\kern-0.25ex\right\vert}}

\usepackage{epstopdf}
\newtheorem{theorem}{Theorem}[section]
\newtheorem{lemma}{Lemma}[section]
\newtheorem{remark}{Remark}[section]

\newtheorem{corollary}{Corollary}[section]

\newtheorem{proposition}{Proposition}[section]
\numberwithin{equation}{section}

\begin{document}
\title{Tightening Bounds on the Numerical Radius for Hilbert Space Operators}
\author{Maryam Jalili and Hamid Reza Moradi}

\begin{abstract}
Let $S$ be a bounded linear operator on a Hilbert space. We show that if $S$ is accretive (resp. dissipative the sense that $\frac{S-{{S}^{*}}}{2i}$ is positive) in the sense that $\frac{S+{{S}^{*}}}{2}$ is positive, then
\[\frac{\sqrt{3}}{3}\left\| S \right\|\le \omega \left( S \right),\]
where $\left\| \cdot \right\| $ and $\omega \left( \cdot \right)$ denote the operator norm and the numerical radius, respectively.
\end{abstract}

\maketitle
\pagestyle{myheadings}
\markboth{\centerline{}}
{\centerline{}}
\bigskip
\bigskip

\section{Introduction}
Let $\mathbb{H}$ be an arbitrary Hilbert space, endowed with the inner product $\left<\cdot,\cdot\right>$ and induced norm $\|\cdot\|.$ The notation $\mathbb{B}(\mathbb{H})$ will be used to denote the $C^*$-algebra of all bounded linear operators on $\mathbb{H}$. Upper case letters will be used to denote the element of $\mathbb{B}(\mathbb{H})$. For $S\in\mathbb{B}(\mathbb{H})$, the adjoint operator $S^*$ is the operator defined by $\left<Sx,y\right>=\left<x,S^*y\right>$ for $x,y\in\mathbb{H}$, and the operator norm of $S$ is defined by $\|S\|=\sup_{\|x\|=1}\|Sx\|$. If an operator $S\in\mathbb{B}(\mathbb{H})$ satisfies $\left<Sx,x\right>\geq 0$ for all $x\in\mathbb{H}$, it will be a positive operator.  In this case, we write $S\ge O$. Here $\left| S \right|$ stands for the positive operator ${{\left( {{S}^{*}}S \right)}^{\frac{1}{2}}}$. Related to the operator norm, the numerical radius of $S\in\mathbb{B}(\mathbb{H})$ is defined by $\omega(S)=\sup_{\|x\|=1}|\left<Sx,x\right>|$. This latter quantity defines a norm on $\mathbb{B}(\mathbb{H})$ that is equivalent to the operator norm, where we have the equivalence
\begin{equation}\label{eq_equiv}
\frac{1}{2}\|S\|\leq \omega(S)\leq \|S\|, S\in\mathbb{B}(\mathbb{H}).
\end{equation} 
Several numerical radius inequalities that improve upon those in \eqref{eq_equiv} have been recently established in \cite{HOM, KMS, SM2023, SMS, S, Yamazaki_Studia_2007}.

Sharpening the second inequality in \eqref{eq_equiv}, Kittaneh showed that \cite{Kittaneh_Studia_2003}
\begin{equation}\label{eq_intro_kitt_1}
\omega \left( S \right)\le \frac{1}{2}\left( \left\| S \right\|+{{\left\| {{S}^{2}} \right\|}^{\frac{1}{2}}} \right).
\end{equation}
Let $r\left( \cdot \right)$ denote the spectral radius. It has been shown in \cite[Theorem 2.1]{BP} that
\begin{equation}\label{9}
\omega \left( S \right)\le \frac{1}{2}\left( \left\| S \right\|+\sqrt{r\left( \left| S \right|\left| {{S}^{*}} \right| \right)} \right).
\end{equation}
Of course \eqref{9} improves \eqref{eq_intro_kitt_1} since 
\[\sqrt{r\left( \left| S \right|\left| {{S}^{*}} \right| \right)}\le {{\left\| \;\left| S \right|\left| {{S}^{*}} \right|\; \right\|}^{\frac{1}{2}}}={{\left\| {{S}^{2}} \right\|}^{\frac{1}{2}}}.\]
A further upper estimate of the numerical radius can be found in \cite[Corollary 3.3]{Heydarbeygi}, presented as follows
\begin{equation}\label{29}
{{\omega }^{2}}\left( S \right)\le \frac{1}{4}\left\| {{\left| S \right|}^{2}}+{{\left| {{S}^{*}} \right|}^{2}} \right\|+\frac{1}{2}\omega \left( \left| S \right|\left| {{S}^{*}} \right| \right).
\end{equation}

Let $S\in\mathbb{B}(\mathbb{H})$. $S$ is accretive (resp. dissipative) if in its Cartesian decomposition $S=\Re S+i\Im S$, $\Re S=\frac{S+{{S}^{*}}}{2}$ (resp. $\Im S=\frac{S-{{S}^{*}}}{2i}$) is positive and $S$ is accretive-dissipative if both $\Re T$ and $\Im S$ are positive. 

The first inequality in \eqref{eq_equiv} has been improved considerably in \cite{MS}. Indeed, it has been shown in \cite[Theorem 2.6]{MS} that if $S$ is an accretive-dissipative operator, then
\begin{equation}\label{12}
\frac{\sqrt{2}}{2}\left\| S \right\|\le \omega \left( S \right).
\end{equation}

As we proceed in this paper, the following basic facts will be used (see, e.g., \cite[Ineq. (2.3) and Ineq. (2.4)]{OM}):
\begin{equation}\label{21}
\left\| \Re S \right\|,\left\| \Im S \right\|\le \omega \left( S \right),
\end{equation}
for any $S\in\mathbb{B}(\mathbb{H})$.

The paper includes two main results, namely Corollary \ref{10} and Theorem \ref{8}. Our first main result concerns an improvement for \eqref{9}, which is as follows: 
\begin{equation}\label{11}
\omega \left( S \right)\le \frac{1}{2}\max \left\{ \left\| \left| {{S}^{*}} \right|+{{\left( {{\left| {{S}^{*}} \right|}^{\frac{1}{2}}}\left| S \right|\left| {{S}^{*}} \right|^{\frac{1}{2}} \right)}^{\frac{1}{2}}} \right\|,\left\| \left| S \right|+{{\left( {{\left| S \right|}^{\frac{1}{2}}}\left| {{S}^{*}} \right|\left| S \right|^{\frac{1}{2}} \right)}^{\frac{1}{2}}} \right\| \right\}.
\end{equation}
To achieve \eqref{11}, we propose an upper bound for $2 \times 2$ operator matrices $\omega \left( \left[ \begin{matrix}
   O & S  \\
   {{T}^{*}} & O  \\
\end{matrix} \right] \right)$. 

Our second main result is similar to \eqref{11}, with the difference that the condition involving accretive-dissipative has been replaced by a weaker condition involving accretive (or dissipative). The derived bound leads to an improvement for the first inequality in \eqref{eq_equiv}. In fact, we will demonstrate that if $S\in \mathbb B\left( \mathbb H \right)$ is such that $S$ is accretive (or dissipative), then
\[\frac{\sqrt{3}}{3}\left\| S \right\|\le \omega \left( S \right).\]

To prove our main result, we need some lemmas. Kittaneh proved the first lemma in \cite[Corollary 1]{kmy1}.
\begin{lemma}\label{k1}
 Let $A,B\in \mathbb B\left( \mathbb H \right)$ be positive. Then
\begin{equation}
\left\| A+B \right\|\le \max \left\{ \left\| A+\left| {{B}^{\frac{1}{2}}}{{A}^{\frac{1}{2}}} \right| \right\|,\left\| B+\left| {{A}^{\frac{1}{2}}}{{B}^{\frac{1}{2}}} \right| \right\| \right\}.
\end{equation}	
\end{lemma}

The second lemma represents an improvement of the basic triangle inequality (for more details, see \cite[p. 27]{batia-princ}).
\begin{lemma}\label{2}
Let $A,B\in \mathbb B\left( \mathbb H \right)$ be self-adjoint. Then
\[\left\| A+B \right\|\le \left\| \;\left| A \right|+\left| B \right| \;\right\|.\]
\end{lemma}

It should be interesting to notice that the above lemma is also true when $A, B$ are two normal matrices. This result has been proved in \cite[Corollary 3.2]{SGM}.

The reader is referred to the proof of the third lemma in \cite[(4.6)]{HKS}.
\begin{lemma}\label{3}
Let $A,B\in \mathbb B\left( \mathbb H \right)$. Then
\[\frac{1}{2}\underset{\theta \in \mathbb{R}}{\mathop{\sup }}\,\left\| A+{{e}^{i\theta }}B \right\|=\omega \left( \left[ \begin{matrix}
   O & A  \\
   {{B}^{*}} & O  \\
\end{matrix} \right] \right).\]
\end{lemma}

Kittaneh also proved the following three lemmas in 2008 and 2004.
\begin{lemma}\label{14}
\cite[Theorem 1]{K} Let $S,T\in \mathbb B\left( \mathbb H \right)$ be such that $S$ or $T$ is positive. Then
\[\left\| ST-TS \right\|\le \left\| S \right\|\left\| T \right\|.\]
\end{lemma}

\begin{lemma}\label{17}
\cite[Corollary 1]{Kittaneh2004} Let $A,B\in \mathbb B\left( \mathbb H \right)$ be positive. Then
\[\max \left\{ \left\| A \right\|,\left\| B \right\| \right\}-\left\| {{A}^{\frac{1}{2}}}{{B}^{\frac{1}{2}}} \right\|\le \left\| A-B \right\|.\]
\end{lemma}

\begin{lemma}\label{22}
\cite{K2008} Let $A,B\in \mathbb B\left( \mathbb H \right)$ and let $A$ be normal with the Cartesian decomposition $A=\Re A+i\Im A$. Then, for every unitarily invariant norm,
	\[{{\vertiii{ AB-BA }}}\le \sqrt{{{\left\| \Re A \right\|}^{2}}+{{\left\| \Im A \right\|}^{2}}}\;{{\vertiii{ B }}},\]
when $\Re A,\Im A\ge O$. 
\end{lemma}

The following lemma is recognized as the mixed Cauchy-Schwarz inequality in the literature. For its proof, see \cite{Kittaneh1988}.
\begin{lemma}\label{27}
Let $A\in \mathbb B\left( \mathbb H \right)$. If $f$, $g$ are nonnegative continuous functions on $\left[ 0,\infty  \right)$ satisfying $f\left( t \right)g\left( t \right)=t\left( t\ge 0 \right)$, then
\[\left| \left\langle Ax,y \right\rangle  \right|\le \sqrt{\left\langle {{f}^{2}}\left( \left| A \right| \right)x,x \right\rangle \left\langle {{g}^{2}}\left( \left| {{A}^{*}} \right| \right)x,x \right\rangle };\;\left( x,y\in \mathbb H \right).\]
\end{lemma}

The final lemma we include here is the Hölder-McCarthy inequality \cite[Theorem 1.2]{Micic}.
\begin{lemma}\label{28}
Let $A\in \mathbb B\left( \mathbb H \right)$ be positive. Then for any $r\ge 1$
\[{{\left\langle Ax,x \right\rangle }^{r}}\le \left\langle {{A}^{r}}x,x \right\rangle ;\;\left( x\in \mathbb H,\left\| x \right\|=1 \right).\]
\end{lemma}

\section{Results}
We begin this section with the following result, which proposes an upper bound for $\omega \left( S \right)$.
\begin{theorem}
Let $S\in \mathbb B\left( \mathbb H \right)$. If $f$, $g$ are nonnegative continuous functions on $\left[ 0,\infty  \right)$ satisfying $f\left( t \right)g\left( t \right)=t\left( t\ge 0 \right)$, then
\[{{\omega }^{2}}\left( S \right)\le \frac{1}{4}\left\| \Re\left( {{f}^{4}}\left( \left| S \right| \right)+{{g}^{4}}\left( \left| {{S}^{*}} \right| \right)+2{{f}^{2}}\left( \left| S \right| \right){{g}^{2}}\left( \left| {{S}^{*}} \right| \right) \right) \right\|.\]
\end{theorem}
\begin{proof}
Let $x\in \mathbb H$ be a unit vector. We have
\[\begin{aligned}
  & {{\left| \left\langle Sx,x \right\rangle  \right|}^{2}} \\ 
 & \le \left\langle {{f}^{2}}\left( \left| S \right| \right)x,x \right\rangle \left\langle {{g}^{2}}\left( \left| {{S}^{*}} \right| \right)x,x \right\rangle  \\ 
 & \le {{\left( \frac{1}{2}\left( \left\langle {{f}^{2}}\left( \left| S \right| \right)x,x \right\rangle +\left\langle {{g}^{2}}\left( \left| {{S}^{*}} \right| \right)x,x \right\rangle  \right) \right)}^{2}} \\ 
 & =\frac{1}{4}{{\left\langle \left( {{f}^{2}}\left( \left| S \right| \right)+{{g}^{2}}\left( \left| {{S}^{*}} \right| \right) \right)x,x \right\rangle }^{2}} \\ 
 & \le \frac{1}{4}\left\langle {{\left( {{f}^{2}}\left( \left| S \right| \right)+{{g}^{2}}\left( \left| {{S}^{*}} \right| \right) \right)}^{2}}x,x \right\rangle  \\ 
 & =\frac{1}{4}\left\langle \left( {{f}^{4}}\left( \left| S \right| \right)+{{g}^{4}}\left( \left| {{S}^{*}} \right| \right)+{{f}^{2}}\left( \left| S \right| \right){{g}^{2}}\left( \left| {{S}^{*}} \right| \right)+{{g}^{2}}\left( \left| {{S}^{*}} \right| \right){{f}^{2}}\left( \left| S \right| \right) \right)x,x \right\rangle  \\ 
 & =\frac{1}{4}\left\langle \left( {{f}^{4}}\left( \left| S \right| \right)+{{g}^{4}}\left( \left| {{S}^{*}} \right| \right)+2\Re\left( {{f}^{2}}\left( \left| S \right| \right){{g}^{2}}\left( \left| {{S}^{*}} \right| \right) \right) \right)x,x \right\rangle  \\ 
 & =\frac{1}{4}\left\langle \Re\left( {{f}^{4}}\left( \left| S \right| \right)+{{g}^{4}}\left( \left| {{S}^{*}} \right| \right)+2{{f}^{2}}\left( \left| S \right| \right){{g}^{2}}\left( \left| {{S}^{*}} \right| \right) \right)x,x \right\rangle   
\end{aligned}\]
where the initial inequality is derived from Lemma \ref{27}, the subsequent inequality results from the arithmetic-geometric mean inequality and the third inequality follows from Lemma \ref{28}. Indeed, we have demonstrated that
\[{{\left| \left\langle Sx,x \right\rangle  \right|}^{2}}\le \frac{1}{4}\left\langle \Re\left( {{f}^{4}}\left( \left| S \right| \right)+{{g}^{4}}\left( \left| {{S}^{*}} \right| \right)+2{{f}^{2}}\left( \left| S \right| \right){{g}^{2}}\left( \left| {{S}^{*}} \right| \right) \right)x,x \right\rangle \]
for any unit vector $x\in \mathbb H$. By taking the supremum over all unit vectors $x\in \mathbb H$, we get
\[{{\omega }^{2}}\left( S \right)\le \frac{1}{4}\left\| \Re\left( {{f}^{4}}\left( \left| S \right| \right)+{{g}^{4}}\left( \left| {{S}^{*}} \right| \right)+2{{f}^{2}}\left( \left| S \right| \right){{g}^{2}}\left( \left| {{S}^{*}} \right| \right) \right) \right\|,\]
as required.
\end{proof}

\begin{corollary}
Let $S\in \mathbb B\left( \mathbb H \right)$. Then for any $0 \le t \le 1$
\[{{\omega }^{2}}\left( S \right)\le \frac{1}{4}\left\| \Re\left( {{\left| S \right|}^{4\left( 1-t \right)}}+{{\left| {{S}^{*}} \right|}^{4t}}+2{{\left| S \right|}^{2\left( 1-t \right)}}{{\left| {{S}^{*}} \right|}^{2t}} \right) \right\|.\]	
In particular,
\begin{equation}\label{30}
{{\omega }^{2}}\left( S \right)\le \frac{1}{4}\left\| \Re\left( {{\left| S \right|}^{2}}+{{\left| {{S}^{*}} \right|}^{2}}+2\left| S \right|\left| {{S}^{*}} \right| \right) \right\|.
\end{equation}
\end{corollary}

\begin{remark}\label{15}
The constant $\frac{1}{4}$ is best possible in \eqref{30}. Indeed, if we assume that \eqref{13} holds with a constant $c>0$, i.e.,
\[{{\omega }^{2}}\left( S \right)\le c\left\| \Re\left( {{\left| S \right|}^{2}}+{{\left| {{S}^{*}} \right|}^{2}}+2\left| S \right|\left| {{S}^{*}} \right| \right) \right\|,\]
 then selecting $S$ as a normal operator (satisfying ${{S}^{*}}S=S{{S}^{*}}$) and utilizing the property that for normal operator $\omega \left( S \right)=\left\| S \right\|$, we conclude that $1\le 4c$ thereby establishing the optimality of the constant.
\end{remark}

\begin{remark}
The fact that
\[\begin{aligned}
   {{\omega }^{2}}\left( S \right)&\le \frac{1}{4}\left\| \Re\left( {{\left| S \right|}^{2}}+{{\left| {{S}^{*}} \right|}^{2}}+2\left| S \right|\left| {{S}^{*}} \right| \right) \right\| \\ 
 & \le \frac{1}{4}\omega \left( {{\left| S \right|}^{2}}+{{\left| {{S}^{*}} \right|}^{2}}+2\left| S \right|\left| {{S}^{*}} \right| \right) \\ 
 & \le \frac{1}{4}\omega \left( {{\left| S \right|}^{2}}+{{\left| {{S}^{*}} \right|}^{2}} \right)+\frac{1}{2}\omega \left( \left| S \right|\left| {{S}^{*}} \right| \right) \\ 
 & =\frac{1}{4}\left\| {{\left| S \right|}^{2}}+{{\left| {{S}^{*}} \right|}^{2}} \right\|+\frac{1}{2}\omega \left( \left| S \right|\left| {{S}^{*}} \right| \right)  
\end{aligned}\]
is apparent. This implies that our result improves upon \eqref{29}.
\end{remark}

We adopt some strategies from \cite[Theorem 2.1]{SMK} in the following result.
\begin{theorem}\label{24}
Let $S, T\in \mathbb B\left( \mathbb H \right)$. Then 
\[\omega \left( \left[ \begin{matrix}
   O & S  \\
   {{T}^{*}} & O  \\
\end{matrix} \right] \right)\le \frac{1}{2}\max \left\{ \alpha ,\beta  \right\},\]
where
\[\alpha =\max \left\{ \left\| \left| {{S}^{*}} \right|+{{\left( {{\left| {{S}^{*}} \right|}^{\frac{1}{2}}}\left| {{T}^{*}} \right|{{\left| {{S}^{*}} \right|}^{\frac{1}{2}}} \right)}^{\frac{1}{2}}} \right\|,\left\| \left| S \right|+{{\left( {{\left| S \right|}^{\frac{1}{2}}}\left| T \right|{{\left| S \right|}^{\frac{1}{2}}} \right)}^{\frac{1}{2}}} \right\| \right\},\]
and
\[\beta =\max \left\{ \left\| \left| {{T}^{*}} \right|+{{\left( {{\left| {{T}^{*}} \right|}^{\frac{1}{2}}}\left| {{S}^{*}} \right|{{\left| {{T}^{*}} \right|}^{\frac{1}{2}}} \right)}^{\frac{1}{2}}} \right\|,\left\| \left| T \right|+{{\left( {{\left| T \right|}^{\frac{1}{2}}}\left| S \right|{{\left| T \right|}^{\frac{1}{2}}} \right)}^{\frac{1}{2}}} \right\| \right\}.\]
\end{theorem}
\begin{proof}
Applying Lemma \ref{k1} on the right-hand side of Lemma \ref{2} implies the relation
\begin{equation}\label{k2}
\left\| A+B \right\|\le \max \left\{ \left\| \left| A \right|+\left| {{\left| B \right|}^{\frac{1}{2}}}{{\left| A \right|}^{\frac{1}{2}}} \right| \right\|,\left\| \left| B \right|+\left| {{\left| A \right|}^{\frac{1}{2}}}{{\left| B \right|}^{\frac{1}{2}}} \right| \right\| \right\}
\end{equation}
that holds for self-adjoint operators  $A,B\in \mathbb B\left( \mathbb H \right)$.
Now, let $S,T\in \mathbb B\left( \mathbb H \right)$ be a pair of arbitrary operators and let $A=\left[ \begin{matrix}
   O & S  \\
   {{S}^{*}} & O  \\
\end{matrix} \right]$, $B=\left[ \begin{matrix}
   O & T  \\
   {{T}^{*}} & O  \\
\end{matrix} \right]$.
Obviously, $A$ and $B$ are self-adjoint, so utilizing \eqref{k2}, we have that
\[\begin{aligned}
   \left\| A+B \right\|&=\left\| \left[ \begin{matrix}
   O & S+T  \\
   {{S}^{*}}+{{T}^{*}} & O  \\
\end{matrix} \right] \right\| \\
 & \le \max \left\{ \left\| \left| \left[ \begin{matrix}
   O & S  \\
   {{S}^{*}} & O  \\
\end{matrix} \right] \right|+\left| {{\left| \left[ \begin{matrix}
   O & T  \\
   {{T}^{*}} & O  \\
\end{matrix} \right] \right|}^{\frac{1}{2}}}{{\left| \left[ \begin{matrix}
   O & S  \\
   {{S}^{*}} & O  \\
\end{matrix} \right] \right|}^{\frac{1}{2}}} \right| \right\|,\right. \\
 &\qquad \left. \left\| \left| \left[ \begin{matrix}
   O & T  \\
   {{T}^{*}} & O  \\
\end{matrix} \right] \right|+\left| {{\left| \left[ \begin{matrix}
   O & S  \\
   {{S}^{*}} & O  \\
\end{matrix} \right] \right|}^{\frac{1}{2}}}{{\left| \left[ \begin{matrix}
   O & T  \\
   {{T}^{*}} & O  \\
\end{matrix} \right] \right|}^{\frac{1}{2}}} \right| \right\| \right\} \\
 & =\max \left\{ \left\| \left[ \begin{matrix}
   \left| {{S}^{*}} \right|+{{\left( {{\left| {{S}^{*}} \right|}^{\frac{1}{2}}}\left| {{T}^{*}} \right|{{\left| {{S}^{*}} \right|}^{\frac{1}{2}}} \right)}^{\frac{1}{2}}} & O  \\
   O & \left| S \right|+{{\left( {{\left| S \right|}^{\frac{1}{2}}}\left| T \right|{{\left| S \right|}^{\frac{1}{2}}} \right)}^{\frac{1}{2}}}  \\
\end{matrix} \right] \right\|, \right. \\
 &\qquad \left. \left\| \left[ \begin{matrix}
   \left| {{T}^{*}} \right|+{{\left( {{\left| {{T}^{*}} \right|}^{\frac{1}{2}}}\left| {{S}^{*}} \right|{{\left| {{T}^{*}} \right|}^{\frac{1}{2}}} \right)}^{\frac{1}{2}}} & O  \\
   O & \left| T \right|+{{\left( {{\left| T \right|}^{\frac{1}{2}}}\left| S \right|{{\left| T \right|}^{\frac{1}{2}}} \right)}^{\frac{1}{2}}}  \\
\end{matrix} \right] \right\| \right\} \\
 & =\max \left\{ \alpha ,\beta  \right\},
\end{aligned}\]
where
\[\alpha =\max \left\{ \left\| \left| {{S}^{*}} \right|+{{\left( {{\left| {{S}^{*}} \right|}^{\frac{1}{2}}}\left| {{T}^{*}} \right|{{\left| {{S}^{*}} \right|}^{\frac{1}{2}}} \right)}^{\frac{1}{2}}} \right\|,\left\| \left| S \right|+{{\left( {{\left| S \right|}^{\frac{1}{2}}}\left| T \right|{{\left| S \right|}^{\frac{1}{2}}} \right)}^{\frac{1}{2}}} \right\| \right\},\]
and
\[\beta =\max \left\{ \left\| \left| {{T}^{*}} \right|+{{\left( {{\left| {{T}^{*}} \right|}^{\frac{1}{2}}}\left| {{S}^{*}} \right|{{\left| {{T}^{*}} \right|}^{\frac{1}{2}}} \right)}^{\frac{1}{2}}} \right\|,\left\| \left| T \right|+{{\left( {{\left| T \right|}^{\frac{1}{2}}}\left| S \right|{{\left| T \right|}^{\frac{1}{2}}} \right)}^{\frac{1}{2}}} \right\| \right\}.\]
Here, the first equality sign holds due to
$$
\left[ \begin{matrix}
   |S^{*}|& O  \\
   O & |S|  \\
\end{matrix} \right]^{2}=\left[ \begin{matrix}
   |S^{*}|^{2}& O  \\
   O & |S|^{2}  \\
\end{matrix} \right]=\left[ \begin{matrix}
   SS^{*}& O  \\
   O & S^{*}S  \\
\end{matrix} \right]=\left[ \begin{matrix}
   O& S  \\
   S^{*} & O  \\
\end{matrix} \right]^{2},
$$
while the second  sign is a consequence of the relation
\begin{equation}\label{pomoc}\left\| \left[ \begin{matrix}
   O & X  \\
   Y & O  \\
\end{matrix} \right] \right\|=\max \left\{ \left\| X \right\|,\left\| Y \right\| \right\},
\end{equation}
valid for $X,Y\in \mathbb B\left( \mathbb H \right)$. Finally, another application of \eqref{pomoc}
provides
$$
 \left\| A+B \right\|=\left\| \left[ \begin{matrix}
   O & S+T  \\
   {{S}^{*}}+{{T}^{*}} & O  \\
\end{matrix} \right] \right\|=\max \left\{ \left\| S+T \right\|,\left\| (S+T)^{*} \right\| \right\}=\left\| S+T \right\|,
$$
which in turn implies that
\begin{equation}\label{1}
\left\| S+T \right\|\le \max \left\{ \alpha ,\beta  \right\}.
\end{equation}
If we replace $T$ by ${{e}^{i\theta }}T$, in \eqref{1}, we infer that
\[\left\| S+{{e}^{i\theta }}T \right\|\le \max \left\{ \alpha ,\beta  \right\}.\]
Now, by taking the supremum over $\theta \in \mathbb{R}$, and use Lemma \ref{3}, we obtain
\[\omega \left( \left[ \begin{matrix}
   O & S  \\
   {{T}^{*}} & O  \\
\end{matrix} \right] \right)\le \frac{1}{2}\max \left\{ \alpha ,\beta  \right\},\]
as required.
\end{proof}

\begin{remark}
Here, it is pertinent to point out that Theorem \ref{24} improves upon \cite[Theorem 2.1]{SMK}, because 
\[\left\| A+B \right\|\le 2\omega \left( \left[ \begin{matrix}
   O & A  \\
   {{B}^{*}} & O  \\
\end{matrix} \right] \right).\]
To learn more about the above inequality, see the proof of Theorem 4.6 in \cite{HKS}.
\end{remark}

The case ${{T}^{*}}=S$ in Theorem \ref{24} introduces one of the main results of this paper. In the forthcoming result, we make use of the following evidence:
\[\omega \left( \left[ \begin{matrix}
   O & S  \\
   S & O  \\
\end{matrix} \right] \right)=\omega \left( S \right).\]
\begin{corollary}\label{10}
Let $S\in \mathbb B\left( \mathbb H \right)$. Then 
\[\begin{aligned}
   \omega \left( S \right)&\le \frac{1}{2}\max \left\{ \left\| \left| {{S}^{*}} \right|+{{\left( {{\left| {{S}^{*}} \right|}^{\frac{1}{2}}}\left| S \right|\left| {{S}^{*}} \right| ^{\frac{1}{2}}\right)}^{\frac{1}{2}}} \right\|,\left\| \left| S \right|+{{\left( {{\left| S \right|}^{\frac{1}{2}}}\left| {{S}^{*}} \right|\left| S \right| ^{\frac{1}{2}}\right)}^{\frac{1}{2}}} \right\| \right\} \\ 
 & \le \frac{1}{2}\left( \left\| S \right\|+\sqrt{r\left( \left| S \right|\left| {{S}^{*}} \right| \right)} \right).  
\end{aligned}\]
\end{corollary}
\begin{proof}
Letting ${{T}^{*}}=S$, in Theorem \ref{24}, we have
\begin{equation}\label{13}
\omega \left( S \right)\le \frac{1}{2}\max \left\{ \left\| \left| {{S}^{*}} \right|+{{\left( {{\left| {{S}^{*}} \right|}^{\frac{1}{2}}}\left| S \right|\left| {{S}^{*}} \right|^{\frac{1}{2}} \right)}^{\frac{1}{2}}} \right\|,\left\| \left| S \right|+{{\left( {{\left| S \right|}^{\frac{1}{2}}}\left| {{S}^{*}} \right|\left| S \right|^{\frac{1}{2}} \right)}^{\frac{1}{2}}} \right\| \right\}.
\end{equation}
On the other hand, we know that
\[\begin{aligned}
  & \max \left\{ \left\| \left| {{S}^{*}} \right|+{{\left( {{\left| {{S}^{*}} \right|}^{\frac{1}{2}}}\left| S \right|\left| {{S}^{*}} \right|^{\frac{1}{2}} \right)}^{\frac{1}{2}}} \right\|,\left\| \left| S \right|+{{\left( {{\left| S \right|}^{\frac{1}{2}}}\left| {{S}^{*}} \right|\left| S \right|^{\frac{1}{2}} \right)}^{\frac{1}{2}}} \right\| \right\} \\ 
 & \le \max \left\{ \left\| \;\left| {{S}^{*}} \right| \;\right\|+\left\| {{\left( {{\left| {{S}^{*}} \right|}^{\frac{1}{2}}}\left| S \right|\left| {{S}^{*}} \right| ^{\frac{1}{2}}\right)}^{\frac{1}{2}}} \right\|,\left\| \;\left| S \right| \;\right\|+\left\| {{\left( {{\left| S \right|}^{\frac{1}{2}}}\left| {{S}^{*}} \right|\left| S \right|^{\frac{1}{2}} \right)}^{\frac{1}{2}}} \right\| \right\} \\ 
 & =\max \left\{ \left\| S \right\|+\left\| {{\left( {{\left| {{S}^{*}} \right|}^{\frac{1}{2}}}\left| S \right|\left| {{S}^{*}} \right| ^{\frac{1}{2}}\right)}^{\frac{1}{2}}} \right\|,\left\| S \right\|+\left\| {{\left( {{\left| S \right|}^{\frac{1}{2}}}\left| {{S}^{*}} \right|\left| S \right| ^{\frac{1}{2}}\right)}^{\frac{1}{2}}} \right\| \right\} \\ 
 & =\max \left\{ \left\| S \right\|+{{\left\| {{\left| {{S}^{*}} \right|}^{\frac{1}{2}}}\left| S \right|\left| {{S}^{*}} \right| ^{\frac{1}{2}}\right\|}^{\frac{1}{2}}},\left\| S \right\|+{{\left\| {{\left| S \right|}^{\frac{1}{2}}}\left| {{S}^{*}} \right|\left| S \right| ^{\frac{1}{2}}\right\|}^{\frac{1}{2}}} \right\} \\ 
 & =\max \left\{ \left\| S \right\|+{{\left\| {{\left| {{S}^{*}} \right|}^{\frac{1}{2}}}{{\left| S \right|}^{\frac{1}{2}}}{{\left| S \right|}^{\frac{1}{2}}}\left| {{S}^{*}} \right| ^{\frac{1}{2}}\right\|}^{\frac{1}{2}}},\left\| S \right\|+{{\left\| {{\left| S \right|}^{\frac{1}{2}}}{{\left| {{S}^{*}} \right|}^{\frac{1}{2}}}{{\left| {{S}^{*}} \right|}^{\frac{1}{2}}}\left| S \right|^{\frac{1}{2}} \right\|}^{\frac{1}{2}}} \right\} \\ 
 & =\max \left\{ \left\| S \right\|+\left\| {{\left| {{S}^{*}} \right|}^{\frac{1}{2}}}{{\left| S \right|}^{\frac{1}{2}}} \right\|,\left\| S \right\|+\left\| {{\left| S \right|}^{\frac{1}{2}}}{{\left| {{S}^{*}} \right|}^{\frac{1}{2}}} \right\| \right\} \\ 
 & =\left\| S \right\|+\left\| {{\left| {{S}^{*}} \right|}^{\frac{1}{2}}}{{\left| S \right|}^{\frac{1}{2}}} \right\|. 
\end{aligned}\]
Since,
\[\begin{aligned}
   \left\| {{\left| {{S}^{*}} \right|}^{\frac{1}{2}}}{{\left| S \right|}^{\frac{1}{2}}} \right\|&={{\left\| {{\left| {{S}^{*}} \right|}^{\frac{1}{2}}}{{\left| S \right|}^{\frac{1}{2}}}{{\left| S \right|}^{\frac{1}{2}}}{{\left| {{S}^{*}} \right|}^{\frac{1}{2}}} \right\|}^{\frac{1}{2}}} \\ 
 & =\sqrt{r\left( {{\left| {{S}^{*}} \right|}^{\frac{1}{2}}}{{\left| S \right|}^{\frac{1}{2}}}{{\left| S \right|}^{\frac{1}{2}}}{{\left| {{S}^{*}} \right|}^{\frac{1}{2}}} \right)} \\ 
 & =\sqrt{r\left( \left| S \right|\left| {{S}^{*}} \right| \right)},  
\end{aligned}\]
we conclude the desired result.
\end{proof}

\begin{remark}
The constant $\frac{1}{2}$ cannot be improved in \eqref{13}. As its demonstration closely follows that of Remark \ref{15}, we have omitted the details here.
\end{remark}

Theorem \ref{24} can be extended in the following way. In the demonstration of this result, we employ the subsequent inequality, which was initially established in \cite{AS} and is valid for positive operators $A$ and $B$:
\[\left\| {{\left( \frac{A+B}{2} \right)}^{r}} \right\|\le \left\| \frac{{{A}^{r}}+{{B}^{r}}}{2} \right\|,\;r\ge 1.\]
\begin{proposition}\label{4}
Let $S, T\in \mathbb B\left( \mathbb H \right)$. Then for any $r\ge 1$
\[{{\omega }^{r}}\left( \left[ \begin{matrix}
   O & S  \\
   {{T}^{*}} & O  \\
\end{matrix} \right] \right)\le \frac{1}{2}\max \left\{ \lambda ,\mu  \right\}\]
where
\[\lambda =\max \left\{ \left\| {{\left| {{S}^{*}} \right|}^{r}}+{{\left( {{\left| {{S}^{*}} \right|}^{\frac{1}{2}}}\left| {{T}^{*}} \right|{{\left| {{S}^{*}} \right|}^{\frac{1}{2}}} \right)}^{\frac{r}{2}}} \right\|,\left\| {{\left| S \right|}^{r}}+{{\left( {{\left| S \right|}^{\frac{1}{2}}}\left| T \right|{{\left| S \right|}^{\frac{1}{2}}} \right)}^{\frac{r}{2}}} \right\| \right\},\]
and
\[\mu =\max \left\{ \left\| {{\left| {{T}^{*}} \right|}^{r}}+{{\left( {{\left| {{T}^{*}} \right|}^{\frac{1}{2}}}\left| {{S}^{*}} \right|{{\left| {{T}^{*}} \right|}^{\frac{1}{2}}} \right)}^{\frac{r}{2}}} \right\|,\left\| {{\left| T \right|}^{r}}+{{\left( {{\left| T \right|}^{\frac{1}{2}}}\left| S \right|{{\left| T \right|}^{\frac{1}{2}}} \right)}^{\frac{r}{2}}} \right\| \right\}.\]
\end{proposition}
\begin{proof}
We have
\[\begin{aligned}
  & {{\omega }^{r}}\left( \left[ \begin{matrix}
   O & S  \\
   {{T}^{*}} & O  \\
\end{matrix} \right] \right) \\ 
 & \le {{\left( \frac{1}{2}\max \left\{ \left\| \left| {{S}^{*}} \right|+{{\left( {{\left| {{S}^{*}} \right|}^{\frac{1}{2}}}\left| {{T}^{*}} \right|{{\left| {{S}^{*}} \right|}^{\frac{1}{2}}} \right)}^{\frac{1}{2}}} \right\|,\left\| \left| S \right|+{{\left( {{\left| S \right|}^{\frac{1}{2}}}\left| T \right|{{\left| S \right|}^{\frac{1}{2}}} \right)}^{\frac{1}{2}}} \right\| \right\} \right)}^{r}} \\ 
 & =\max \left\{ {{\left\| \frac{\left| {{S}^{*}} \right|+{{\left( {{\left| {{S}^{*}} \right|}^{\frac{1}{2}}}\left| {{T}^{*}} \right|{{\left| {{S}^{*}} \right|}^{\frac{1}{2}}} \right)}^{\frac{1}{2}}}}{2} \right\|}^{r}},{{\left\| \frac{\left| S \right|+{{\left( {{\left| S \right|}^{\frac{1}{2}}}\left| T \right|{{\left| S \right|}^{\frac{1}{2}}} \right)}^{\frac{1}{2}}}}{2} \right\|}^{r}} \right\} \\ 
 & =\max \left\{ \left\| {{\left( \frac{\left| {{S}^{*}} \right|+{{\left( {{\left| {{S}^{*}} \right|}^{\frac{1}{2}}}\left| {{T}^{*}} \right|{{\left| {{S}^{*}} \right|}^{\frac{1}{2}}} \right)}^{\frac{1}{2}}}}{2} \right)}^{r}} \right\|,\left\| {{\left( \frac{\left| S \right|+{{\left( {{\left| S \right|}^{\frac{1}{2}}}\left| T \right|{{\left| S \right|}^{\frac{1}{2}}} \right)}^{\frac{1}{2}}}}{2} \right)}^{r}} \right\| \right\} \\ 
 & \le \max \left\{ \left\| \frac{{{\left| {{S}^{*}} \right|}^{r}}+{{\left( {{\left| {{S}^{*}} \right|}^{\frac{1}{2}}}\left| {{T}^{*}} \right|{{\left| {{S}^{*}} \right|}^{\frac{1}{2}}} \right)}^{\frac{r}{2}}}}{2} \right\|,\left\| \frac{{{\left| S \right|}^{r}}+{{\left( {{\left| S \right|}^{\frac{1}{2}}}\left| T \right|{{\left| S \right|}^{\frac{1}{2}}} \right)}^{\frac{r}{2}}}}{2} \right\| \right\} \\ 
 & =\frac{1}{2}\max \left\{ \left\| {{\left| {{S}^{*}} \right|}^{r}}+{{\left( {{\left| {{S}^{*}} \right|}^{\frac{1}{2}}}\left| {{T}^{*}} \right|{{\left| {{S}^{*}} \right|}^{\frac{1}{2}}} \right)}^{\frac{r}{2}}} \right\|,\left\| {{\left| S \right|}^{r}}+{{\left( {{\left| S \right|}^{\frac{1}{2}}}\left| T \right|{{\left| S \right|}^{\frac{1}{2}}} \right)}^{\frac{r}{2}}} \right\| \right\} \\ 
\end{aligned}\]
i.e.,
\[{{\omega }^{r}}\left( \left[ \begin{matrix}
   O & S  \\
   {{T}^{*}} & O  \\
\end{matrix} \right] \right)\le \frac{1}{2}\max \left\{ \left\| {{\left| {{S}^{*}} \right|}^{r}}+{{\left( {{\left| {{S}^{*}} \right|}^{\frac{1}{2}}}\left| {{T}^{*}} \right|{{\left| {{S}^{*}} \right|}^{\frac{1}{2}}} \right)}^{\frac{r}{2}}} \right\|,\left\| {{\left| S \right|}^{r}}+{{\left( {{\left| S \right|}^{\frac{1}{2}}}\left| T \right|{{\left| S \right|}^{\frac{1}{2}}} \right)}^{\frac{r}{2}}} \right\| \right\}.\]
Similarly, we can show that
\[{{\omega }^{r}}\left( \left[ \begin{matrix}
   O & S  \\
   {{T}^{*}} & O  \\
\end{matrix} \right] \right)\le \frac{1}{2}\max \left\{ \left\| {{\left| {{T}^{*}} \right|}^{r}}+{{\left( {{\left| {{T}^{*}} \right|}^{\frac{1}{2}}}\left| {{S}^{*}} \right|{{\left| {{T}^{*}} \right|}^{\frac{1}{2}}} \right)}^{\frac{r}{2}}} \right\|,\left\| {{\left| T \right|}^{r}}+{{\left( {{\left| T \right|}^{\frac{1}{2}}}\left| S \right|{{\left| T \right|}^{\frac{1}{2}}} \right)}^{\frac{r}{2}}} \right\| \right\}.\]
Combining these two inequalities implies the desired result.
\end{proof}

The scenario where $r=2$, in Proposition \ref{4}, includes interesting results as demonstrated below.
\begin{theorem}\label{5}
Let $S, T\in \mathbb B\left( \mathbb H \right)$. Then
\[{{\omega }^{2}}\left( \left[ \begin{matrix}
   O & S  \\
   {{T}^{*}} & O  \\
\end{matrix} \right] \right)\le \frac{1}{2}\max \left\{ \delta ,\xi  \right\}\]
where
\[\delta =\max \left\{ r\left( \left| {{S}^{*}} \right|\left( \left| {{S}^{*}} \right|+\left| {{T}^{*}} \right| \right) \right),r\left( \left| S \right|\left( \left| S \right|+\left| T \right| \right) \right) \right\},\]
and
\[\xi =\max \left\{ r\left( \left| {{T}^{*}} \right|\left( \left| {{T}^{*}} \right|+\left| {{S}^{*}} \right| \right) \right),r\left( \left| T \right|\left( \left| T \right|+\left| S \right| \right) \right) \right\}.\]
In particular, if $S,T$ are self-adjoint operators, then
\[{{\omega }^{2}}\left( \left[ \begin{matrix}
   O & S  \\
   {{T}} & O  \\
\end{matrix} \right] \right)\le \frac{1}{2}\max \left\{ r\left( \left| S \right|\left( \left| S \right|+\left| T \right| \right) \right),r\left( \left| T \right|\left( \left| T \right|+\left| S \right| \right) \right) \right\}.\]
\end{theorem}
\begin{proof}
It follows from Proposition \ref{4} that
\[\begin{aligned}
  & {{\omega }^{2}}\left( \left[ \begin{matrix}
   O & S  \\
   {{T}^{*}} & O  \\
\end{matrix} \right] \right) \\ 
 & \le \frac{1}{2}\max \left\{ \left\| {{\left| {{S}^{*}} \right|}^{2}}+{{\left| {{S}^{*}} \right|}^{\frac{1}{2}}}\left| {{T}^{*}} \right|{{\left| {{S}^{*}} \right|}^{\frac{1}{2}}} \right\|,\left\| {{\left| S \right|}^{2}}+{{\left| S \right|}^{\frac{1}{2}}}\left| T \right|{{\left| S \right|}^{\frac{1}{2}}} \right\| \right\} \\ 
 & =\frac{1}{2}\max \left\{ \left\| {{\left| {{S}^{*}} \right|}^{\frac{1}{2}}}\left| {{S}^{*}} \right|{{\left| {{S}^{*}} \right|}^{\frac{1}{2}}}+{{\left| {{S}^{*}} \right|}^{\frac{1}{2}}}\left| {{T}^{*}} \right|{{\left| {{S}^{*}} \right|}^{\frac{1}{2}}} \right\|,\left\| {{\left| S \right|}^{\frac{1}{2}}}\left| S \right|{{\left| S \right|}^{\frac{1}{2}}}+{{\left| S \right|}^{\frac{1}{2}}}\left| T \right|{{\left| S \right|}^{\frac{1}{2}}} \right\| \right\} \\ 
 & =\frac{1}{2}\max \left\{ \left\| {{\left| {{S}^{*}} \right|}^{\frac{1}{2}}}\left( \left| {{S}^{*}} \right|+\left| {{T}^{*}} \right| \right){{\left| {{S}^{*}} \right|}^{\frac{1}{2}}} \right\|,\left\| {{\left| S \right|}^{\frac{1}{2}}}\left( \left| S \right|+\left| T \right| \right){{\left| S \right|}^{\frac{1}{2}}} \right\| \right\} \\ 
 & =\frac{1}{2}\max \left\{ r\left( {{\left| {{S}^{*}} \right|}^{\frac{1}{2}}}\left( \left| {{S}^{*}} \right|+\left| {{T}^{*}} \right| \right){{\left| {{S}^{*}} \right|}^{\frac{1}{2}}} \right),r\left( {{\left| S \right|}^{\frac{1}{2}}}\left( \left| S \right|+\left| T \right| \right){{\left| S \right|}^{\frac{1}{2}}} \right) \right\} \\ 
 & =\frac{1}{2}\max \left\{ r\left( \left| {{S}^{*}} \right|\left( \left| {{S}^{*}} \right|+\left| {{T}^{*}} \right| \right) \right),r\left( \left| S \right|\left( \left| S \right|+\left| T \right| \right) \right) \right\} \\ 
\end{aligned}\]
i.e.,
\[{{\omega }^{2}}\left( \left[ \begin{matrix}
   O & S  \\
   {{T}^{*}} & O  \\
\end{matrix} \right] \right)\le \frac{1}{2}\max \left\{ r\left( \left| {{S}^{*}} \right|\left( \left| {{S}^{*}} \right|+\left| {{T}^{*}} \right| \right) \right),r\left( \left| S \right|\left( \left| S \right|+\left| T \right| \right) \right) \right\}.\]
In the same way, we have
\[{{\omega }^{2}}\left( \left[ \begin{matrix}
   O & S  \\
   {{T}^{*}} & O  \\
\end{matrix} \right] \right)\le \frac{1}{2}\max \left\{ r\left( \left| {{T}^{*}} \right|\left( \left| {{T}^{*}} \right|+\left| {{S}^{*}} \right| \right) \right),r\left( \left| T \right|\left( \left| T \right|+\left| S \right| \right) \right) \right\}.\]
Together, these two inequalities imply the desired result.
\end{proof}

A notable specific case of Theorem \ref{5}, offering a refinement of the second inequality in \eqref{eq_equiv}, is presented as follows.
\begin{corollary}
Let $S\in \mathbb B\left( \mathbb H \right)$. Then
\begin{equation}\label{16}
{{\omega }^{2}}\left( S \right)\le \frac{1}{2}\max \left\{ r\left( \left| {{S}^{*}} \right|\left( \left| {{S}^{*}} \right|+\left| S \right| \right) \right),r\left( \left| S \right|\left( \left| {{S}^{*}} \right|+\left| S \right| \right) \right) \right\}.
\end{equation}
\end{corollary}

\begin{remark}
The constant $\frac{1}{2}$ is best possible in \eqref{16}. The verification of this statement is similar to Remark \ref{15}; therefore, its elaboration is omitted.
\end{remark}

Although Theorem \ref{24} has been proved for arbitrary operators, we can establish the following result if the operators are normal.
\begin{proposition}
Let $A, B\in \mathbb B\left( \mathbb H \right)$ be two normal operators. Then
\[\omega \left( \left[ \begin{matrix}
   O & A  \\
   {{B}^{*}} & O  \\
\end{matrix} \right] \right)\le \frac{1}{2}\max \left\{ \left\| \left| A \right|+\left| {{\left| B \right|}^{\frac{1}{2}}}{{\left| A \right|}^{\frac{1}{2}}} \right| \right\|,\left\| \left| B \right|+\left| {{\left| A \right|}^{\frac{1}{2}}}{{\left| B \right|}^{\frac{1}{2}}} \right| \right\| \right\}.\]
In particular,
\[\left\| A \right\|\le \frac{1}{2}\max \left\{ \left\| \left| A \right|+\left| {{\left| {{A}^{*}} \right|}^{\frac{1}{2}}}{{\left| A \right|}^{\frac{1}{2}}} \right| \right\|,\left\| \left| {{A}^{*}} \right|+\left| {{\left| A \right|}^{\frac{1}{2}}}{{\left| {{A}^{*}} \right|}^{\frac{1}{2}}} \right| \right\| \right\}.\]
\end{proposition}
\begin{proof}
It follows from Lemmas \ref{k1} and \ref{2} that
\[\begin{aligned}
  & \left\| A+B \right\| \\ 
 & \le \left\| \left| A \right|+\left| B \right| \right\| \\ 
 & \le \max \left\{ \left\| \left| A \right|+\left| {{\left| B \right|}^{\frac{1}{2}}}{{\left| A \right|}^{\frac{1}{2}}} \right| \right\|,\left\| \left| B \right|+\left| {{\left| A \right|}^{\frac{1}{2}}}{{\left| B \right|}^{\frac{1}{2}}} \right| \right\| \right\} \\ 
\end{aligned}\]
i.e.,
\begin{equation}\label{7}
\left\| A+B \right\|\le \max \left\{ \left\| \left| A \right|+\left| {{\left| B \right|}^{\frac{1}{2}}}{{\left| A \right|}^{\frac{1}{2}}} \right| \right\|,\left\| \left| B \right|+\left| {{\left| A \right|}^{\frac{1}{2}}}{{\left| B \right|}^{\frac{1}{2}}} \right| \right\| \right\}.
\end{equation}
If we replace $B$ by ${{e}^{i\theta }}B$, and then taking the supremum over $\theta \in \mathbb{R}$, we obtain
\[\omega \left( \left[ \begin{matrix}
   O & A  \\
   {{B}^{*}} & O  \\
\end{matrix} \right] \right)\le \frac{1}{2}\max \left\{ \left\| \left| A \right|+\left| {{\left| B \right|}^{\frac{1}{2}}}{{\left| A \right|}^{\frac{1}{2}}} \right| \right\|,\left\| \left| B \right|+\left| {{\left| A \right|}^{\frac{1}{2}}}{{\left| B \right|}^{\frac{1}{2}}} \right| \right\| \right\},\]
as required.

The second inequality can be obtained from the first inequality by setting $B^* = A$ and using the fact that, for normal operators, the numerical radius equals the operator norm.
\end{proof}

The following theorem gives us a new bound for the operator norm.
\begin{theorem}\label{6}
Let $S\in \mathbb B\left( \mathbb H \right)$. Then
\[\left\| S \right\|\le \omega \left( S \right)+\sqrt{r\left( \left| \Re S \right|\left| \Im S \right| \right)}.\]
\end{theorem}
\begin{proof}
We know that if $T$ is a normal operator, then $iT$ is also a normal operator. So, from \eqref{7}, we can write
\[\left\| A+iB \right\|\le \max \left\{ \left\| \left| A \right|+\left| {{\left| B \right|}^{\frac{1}{2}}}{{\left| A \right|}^{\frac{1}{2}}} \right| \right\|,\left\| \left| B \right|+\left| {{\left| A \right|}^{\frac{1}{2}}}{{\left| B \right|}^{\frac{1}{2}}} \right| \right\| \right\}.\]
Let $S=A+iB$ be the Cartesian decomposition of $S$, where $A=\Re S$ and $B=\Im S$. We obtain
\[\begin{aligned}
   \left\| S \right\|&\le \max \left\{ \left\| \left| \Re S \right|+\left| {{\left| \Im S \right|}^{\frac{1}{2}}}{{\left| \Re S \right|}^{\frac{1}{2}}} \right| \right\|,\left\| \left| \Im S \right|+\left| {{\left| \Re S \right|}^{\frac{1}{2}}}{{\left| \Im S \right|}^{\frac{1}{2}}} \right| \right\| \right\} \\ 
 & \le \max \left\{ \left\| \;\left| \Re S \right| \;\right\|+\left\| \;\left| {{\left| \Im S \right|}^{\frac{1}{2}}}{{\left| \Re S \right|}^{\frac{1}{2}}} \right|\; \right\|,\left\| \;\left| \Im S \right| \;\right\|+\left\| \;\left| {{\left| \Re S \right|}^{\frac{1}{2}}}{{\left| \Im S \right|}^{\frac{1}{2}}} \right|\; \right\| \right\} \\ 
 & =\max \left\{ \left\| \Re S \right\|+\left\| {{\left| \Im S \right|}^{\frac{1}{2}}}{{\left| \Re S \right|}^{\frac{1}{2}}} \right\|,\left\| \Im S \right\|+\left\| {{\left| \Re S \right|}^{\frac{1}{2}}}{{\left| \Im S \right|}^{\frac{1}{2}}} \right\| \right\} \\ 
 & =\max \left\{ \left\| \Re S \right\|+\left\| {{\left| \Re S \right|}^{\frac{1}{2}}}{{\left| \Im S \right|}^{\frac{1}{2}}} \right\|,\left\| \Im S \right\|+\left\| {{\left| RS \right|}^{\frac{1}{2}}}{{\left| \Im S \right|}^{\frac{1}{2}}} \right\| \right\} \\ 
 & =\max \left\{ \left\| \Re S \right\|,\left\| \Im S \right\| \right\}+\left\| {{\left| \Re S \right|}^{\frac{1}{2}}}{{\left| \Im S \right|}^{\frac{1}{2}}} \right\| \\ 
 \end{aligned}\]
i.e.,
\[\left\| S \right\|\le \max \left\{ \left\| \Re S \right\|,\left\| \Im S \right\| \right\}+\left\| {{\left| \Re S \right|}^{\frac{1}{2}}}{{\left| \Im S \right|}^{\frac{1}{2}}} \right\|.\]
We get from \eqref{21} that
	\[\left\| S \right\|\le \omega \left( S \right)+\left\| {{\left| \Re S \right|}^{\frac{1}{2}}}{{\left| \Im S \right|}^{\frac{1}{2}}} \right\|.\]
On the other hand, we know that
	\[\left\| {{\left| \Re S \right|}^{\frac{1}{2}}}{{\left| \Im S \right|}^{\frac{1}{2}}} \right\|={{\left\| {{\left| \Re S \right|}^{\frac{1}{2}}}{{\left| \Im S \right|}^{\frac{1}{2}}}{{\left| \Im S \right|}^{\frac{1}{2}}}{{\left| \Re S \right|}^{\frac{1}{2}}} \right\|}^{\frac{1}{2}}}=\sqrt{r\left( \left| \Re S \right|\left| \Im S \right| \right)},\]
which in turn implies the desired result.
\end{proof}

\begin{theorem}\label{8}
Let $S\in \mathbb B\left( \mathbb H \right)$ be an accretive operator. Then
\[\frac{\sqrt{3}}{3}\left\| S \right\|\le \omega \left( S \right).\]
The above inequality is also true when $S$ is a dissipative operator.
\end{theorem}
\begin{proof}
According to the assumption $\Re S\ge O$. If we use the triangle inequality for the operator norm and Lemma \ref{14}, we can write
\[\begin{aligned}
   {{\left\| S \right\|}^{2}}&=\left\| S{{S}^{*}} \right\| \\ 
 & =\left\| \left( \Re S+i\Im S \right){{\left( \Re S+i\Im S \right)}^{*}} \right\| \\ 
 & =\left\| \left( \Re S+i\Im S \right)\left( \Re S-i\Im S \right) \right\| \\ 
 & =\left\| {{\left( \Re S \right)}^{2}}+{{\left( \Im S \right)}^{2}}-i\left( \left( \Re S \right)\left( \Im S \right)-\left( \Im S \right)\left( \Re S \right) \right) \right\| \\ 
 & \le\left\| {{\left( \Re S \right)}^{2}}+{{\left( \Im S \right)}^{2}} \right\|+\left\| \left( \Re S \right)\left( \Im S \right)-\left( \Im S \right)\left( \Re S \right) \right\| \\ 
 & \le \left\| {{\left( \Re S \right)}^{2}} \right\|+\left\| {{\left( \Im S \right)}^{2}} \right\|+\left\| \left( \Re S \right)\left( \Im S \right)-\left( \Im S \right)\left( \Re S \right) \right\| \\ 
 & \le {{\left\| \Re S \right\|}^{2}}+{{\left\| \Im S \right\|}^{2}}+\left\| \Re S \right\|\left\| \Im S \right\|  
\end{aligned}\]
i.e.,
\[{{\left\| S \right\|}^{2}}\le {{\left\| \Re S \right\|}^{2}}+{{\left\| \Im S \right\|}^{2}}+\left\| \Re S \right\|\left\| \Im S \right\|.\]
On the other hand, if we apply \eqref{21}, we infer that
\[{{\left\| \Re S \right\|}^{2}}+{{\left\| \Im S \right\|}^{2}}+\left\| \Re S \right\|\left\| \Im S \right\|\le 3{{\omega }^{2}}\left( S \right).\]
Indeed, we have shown that
\[{{\left\| S \right\|}^{2}}\le 3{{\omega }^{2}}\left( S \right),\]
which is equivalent to the desired result.

The other case (when $\Im S\ge O$) can be proven similarly; we delete the details.
\end{proof}

We obtain the well-known inequality 
\[\omega \left( ST \right)\le 4\omega \left( S \right)\omega \left( T \right),\; S,T\in \mathbb B\left( \mathbb H \right)\]
 from \eqref{eq_equiv} and the fact that the operator norm is sub-multiplicative.
We can enhance this approximation under specific conditions, as demonstrated in the final result of this paper.
\begin{corollary}
Let $S,T\in \mathbb B\left( \mathbb H \right)$. If 

\item[(i)] $S$ and $T$ are accretive,\\
or
\item[(ii)] $S$ and $T$ are dissipative, \\
or
\item[(iii)] $S$ is accretive (resp. dissipative) and $T$ is dissipative (resp. accretive),

then
\[\omega \left( ST \right)\le 3\omega \left( S \right)\omega \left( T \right).\]
\end{corollary}
\begin{proof}
It follows from Theorem \ref{8} that
\[\begin{aligned}
   \omega \left( ST \right)&\le \left\| ST \right\| \\ 
 & \le \left\| S \right\|\left\| T \right\| \\ 
 & \le \frac{3}{\sqrt{3}}\omega \left( S \right)\cdot\frac{3}{\sqrt{3}}\omega \left( T \right) \\ 
 & =3\omega \left( S \right)\omega \left( T \right),  
\end{aligned}\]
as required
\end{proof}

\begin{remark}
Let the assumptions of Theorem \ref{8} hold. Since 
\[\begin{aligned}
   {{\left\| S \right\|}^{2}}&=\left\| S{{S}^{*}} \right\| \\ 
 & =\left\| \left( \Re S+i\Im S \right)\left( \Re S-i\Im S \right) \right\| \\ 
 & =\left\| {{\left( \Re S \right)}^{2}}+{{\left( \Im S \right)}^{2}}-i\left( \left( \Re S \right)\left( \Im S \right)-\left( \Im S \right)\left( \Re S \right) \right) \right\| \\ 
 & \le \left\| {{\left( \Re S \right)}^{2}}+{{\left( \Im S \right)}^{2}} \right\|+\left\| \left( \Re S \right)\left( \Im S \right)-\left( \Im S \right)\left( \Re S \right) \right\| \\ 
 & =\left\| {{\left( \frac{S+{{S}^{*}}}{2} \right)}^{2}}+{{\left( \frac{S-{{S}^{*}}}{2i} \right)}^{2}} \right\|+\left\| \left( \Re S \right)\left( \Im S \right)-\left( IS \right)\left( \Re S \right) \right\| \\ 
 & =\frac{1}{2}\left\| S{{S}^{*}}+{{S}^{*}}S \right\|+\left\| \left( \Re S \right)\left( \Im S \right)-\left( \Im S \right)\left( \Re S \right) \right\| \\ 
 & \le \frac{1}{2}\left\| S{{S}^{*}}+{{S}^{*}}S \right\|+\left\| \Re S \right\|\left\| \Im S \right\|,  
\end{aligned}\]
we have
\begin{equation}\label{18}
{{\left\| S \right\|}^{2}}\le \frac{1}{2}\left\| S{{S}^{*}}+{{S}^{*}}S \right\|+\left\| \Re S \right\|\left\| \Im S \right\|.
\end{equation}
\end{remark}

The inequality $\left\| {{S}^{2}} \right\|\le {{\left\| S \right\|}^{2}}$ for any bounded linear operator $S\in \mathbb B\left( \mathbb H \right)$ is a fundamental result in operator theory. As a consequence, it is important to establish upper bounds for the nonnegative difference ${{\left\| S \right\|}^{2}}-\left\| {{S}^{2}} \right\|$ under various assumptions. In the following result, we aim to derive such an inequality.

\begin{theorem}
Let $S\in \mathbb B\left( \mathbb H \right)$ be an accretive operator. Then
\begin{equation}\label{19}
{{\left\| S \right\|}^{2}}-\left\| {{S}^{2}} \right\|\le 2\left\| \Re S \right\|\left\| \Im S \right\|.
\end{equation}
The above inequality is also true when $S$ is a dissipative operator.
\end{theorem}
\begin{proof}
Notice that
\[\begin{aligned}
   \left\| \left( \Re S \right)\left( \Im S \right)-\left( \Im S \right)\left( \Re S \right) \right\|&=\left\| \left( \frac{S+{{S}^{*}}}{2} \right)\left( \frac{S-{{S}^{*}}}{2i} \right)-\left( \frac{S-{{S}^{*}}}{2i} \right)\left( \frac{S+{{S}^{*}}}{2} \right) \right\| \\ 
 & =\frac{1}{2}\left\| {{S}^{*}}S-S{{S}^{*}} \right\|.  
\end{aligned}\]
Therefore, from Lemma \ref{14}, we have
\[\begin{aligned}
   \left\| {{S}^{*}}S-S{{S}^{*}} \right\|&=2\left\| \left( \Re S \right)\left( IS \right)-\left( \Im S \right)\left( \Re S \right) \right\| \\ 
 & \le 2\left\| \Re S \right\|\left\| \Im S \right\|.  
\end{aligned}\]
On the other hand, Lemma \ref{17} ensures that
\[\begin{aligned}
   {{\left\| S \right\|}^{2}}-\left\| {{S}^{2}} \right\|&={{\left\| S \right\|}^{2}}-\left\| {{\left( {{S}^{*}}S \right)}^{\frac{1}{2}}}{{\left( S{{S}^{*}} \right)}^{\frac{1}{2}}} \right\| \\ 
 & =\max \left\{ \left\| {{S}^{*}}S \right\|,\left\| S{{S}^{*}} \right\| \right\}-\left\| {{\left( {{S}^{*}}S \right)}^{\frac{1}{2}}}{{\left( S{{S}^{*}} \right)}^{\frac{1}{2}}} \right\| \\ 
 & \le \left\| {{S}^{*}}S-S{{S}^{*}} \right\|,  
\end{aligned}\]
due to the fact that $\left\| \;\left| A \right|\left| B \right| \;\right\|=\left\| A{{B}^{*}} \right\|$ for any $A,B\in \mathbb B\left( \mathbb H \right)$ (see, e.g., \cite[(8)]{Kittaneh2004}).
\end{proof}

\begin{remark}
It follows from \eqref{18} and \eqref{19} that
\[{{\left\| S \right\|}^{2}}\le \left\| \Re S \right\|\left\| \Im S \right\|+\min \left\{ \left\| \Re S \right\|\left\| \Im S \right\|+\left\| {{S}^{2}} \right\|,\frac{1}{2}\left\| S{{S}^{*}}+{{S}^{*}}S \right\| \right\}.\]
\end{remark}

Let $\vertiii{ \cdot }$ denote any unitarily invariant norm, i.e., a norm with the property that $\vertiii{ UAV }=\vertiii{ A }$ for all $A$, and for all unitary $U$, $V$.
As an extension of the definition of the numerical radius, it has been defined in \cite{Abu-Omar} that 
\begin{equation}\label{23}
{{\omega }_{\vertiii{ \cdot }}}\left( A \right)=\underset{\theta \in \mathbb{R}}{\mathop{\sup }}\,\vertiii{ \Re\left( {{e}^{i\theta }}A \right) }.
\end{equation}
The essential properties and the inequalities associated with this concept have been examined in the same reference.

The following result concerning the unitarily invariant norms may be stated.
\begin{theorem}\label{25}
Let $S\in \mathbb B\left( \mathbb H \right)$.
\begin{itemize}
\item If $S$ is an accretive operator, then
\begin{equation}\label{26}
{{\vertiii{ S{{S}^{*}} }}}\le {{\vertiii{ {{\left( \Re S \right)}^{2}}+{{\left( \Im S \right)}^{2}} }}}+\left\| \Re S \right\|\;{{\vertiii{ \Im S }}}.
\end{equation}
\item If $S$ is a dissipative operator, then
\[{{\vertiii{ S{{S}^{*}} }}}\le {{\vertiii{ {{\left( \Re S \right)}^{2}}+{{\left( \Im S \right)}^{2}} }}}+\left\| \Im S \right\|\;{{\vertiii{ \Re S }}}.\]
\end{itemize}
\end{theorem}
\begin{proof}
Assume that $S$ is accretive. It follows from Lemma \ref{22} that
\[\begin{aligned}
  & {{\vertiii{ S{{S}^{*}} }}} \\ 
 & ={{\vertiii{ \left( \Re S+i\Im S \right)\left( \Re S-i\Im S \right) }}} \\ 
 & ={{\vertiii{ {{\left( \Re S \right)}^{2}}+{{\left( \Im S \right)}^{2}}-i\left( \left( \Re S \right)\left( \Im S \right)-\left( \Im S \right)\left( \Re S \right) \right) }}} \\ 
 & ={{\vertiii{ {{\left( \Re S \right)}^{2}}+{{\left( \Im S \right)}^{2}} }}}+{{\vertiii{ \left( \Re S \right)\left( \Im S \right)-\left( \Im S \right)\left( \Re S \right) }}} \\ 
 & \le {{\vertiii{ {{\left( \Re S \right)}^{2}}+{{\left( \Im S \right)}^{2}} }}}+\left\| \Re S \right\|\;{{\vertiii{ \Im S }}} \\ 
\end{aligned}\]
i.e.,
\[{{\vertiii{ S{{S}^{*}} }}}\le {{\vertiii{ {{\left( \Re S \right)}^{2}}+{{\left( \Im S \right)}^{2}} }}}+\left\| \Re S \right\|\;{{\vertiii{ \Im S }}},\]
which completes the inequality in the first case.

Now, assume that $S$ is dissipative. If we apply the same method as in the above, we can write
\[{{\vertiii{ S{{S}^{*}} }}}\le {{\vertiii{ {{\left( \Re S \right)}^{2}}+{{\left( \Im S \right)}^{2}} }}}+\left\| \Im S \right\|\;{{\vertiii{ \Re S }}},\]
thanks to Lemma \ref{22}.
\end{proof}

We can conclude the following result as a direct consequence of Theorem \ref{25}.
\begin{corollary}
Let $S\in \mathbb B\left( \mathbb H \right)$ be an accretive dissipative operator. Then
\[{{\vertiii{ S{{S}^{*}} }}}\le {{\vertiii{ {{\left( \Re S \right)}^{2}}+{{\left( \Im S \right)}^{2}} }}}+\min \left\{ \left\| \Re S \right\|\;{{\vertiii{ \Im S }}},\left\| \Im S \right\|\;{{\vertiii{ \Re S }}} \right\}.\]
\end{corollary}

Theorem \ref{25} also enables us to prove the following result.
\begin{theorem}
Let $S\in \mathbb B\left( \mathbb H \right)$ be an accretive operator. Then
\[{{\vertiii{ S{{S}^{*}} }}}\le {{\omega }_{\vertiii{ \cdot }}}\left( S \right)\left(2 {{\omega }_{\vertiii{ \cdot }}}\left( S \right)+\omega \left( S \right) \right).\]
The above inequality is also true when $S$ is a dissipative operator.
\end{theorem}
\begin{proof}
Assume that $S$ is accretive. In this case, \eqref{26} holds. It observes from the triangle inequality for the unitarily invariant norm, \eqref{23}, and \eqref{21} that
\[\begin{aligned}
   {{\vertiii{ S{{S}^{*}} }}}&\le {{\vertiii{ {{\left( \Re S \right)}^{2}}+{{\left( \Im S \right)}^{2}} }}}+\left\| \Re S \right\|\;{{\vertiii{ \Im S }}} \\ 
 & \le {{\vertiii{ {{\left( \Re S \right)}^{2}} }}}+{{\vertiii{ {{\left( \Im S \right)}^{2}} }}}+\left\| \Re S \right\|\;{{\vertiii{ \Im S }}} \\ 
 & \le 2\omega _{\vertiii{ \cdot }}^{2}\left( S \right)+{{\omega }_{\vertiii{ \cdot }}}\left( S \right)\;\left\| \Re S \right\| \\ 
 & \le 2\omega _{\vertiii{ \cdot }}^{2}\left( S \right)+{{\omega }_{\vertiii{ \cdot }}}\left( S \right)\;\omega \left( S \right) \\ 
 & ={{\omega }_{\vertiii{ \cdot }}}\left( S \right)\left( 2{{\omega }_{\vertiii{\cdot }}}\left( S \right)+\omega \left( S \right) \right)  
\end{aligned}\]
i.e.,
\[{{\vertiii{ S{{S}^{*}} }}}\le {{\omega }_{\vertiii{ \cdot }}}\left( S \right)\left(2 {{\omega }_{\vertiii{ \cdot }}}\left( S \right)+\omega \left( S \right) \right).\]
The same result also holds when $S$ is dissipative.
\end{proof}

\section*{Declarations}
\begin{itemize}
\item Ethical approval: This declaration is not applicable.
\item Conflict of interest: All authors declare no conflicts of interest.
\item Authors' contribution: The authors have contributed equally to this work.
\item Funding: The authors did not receive funding to accomplish this work.
\item Availability of data and materials: This declaration is not applicable.
\end{itemize}

\vskip 0.7 true cm

\noindent{\tiny (M. Jalili) Department of Mathematics, Neyshabur Branch, Islamic Azad University, Neyshabur, Iran

\noindent \textit{E-mail address:} mrmjlili@iau.ac.ir}

\vskip 0.3 true cm 

\noindent{\tiny (H. R. Moradi) Department of Mathematics, Mashhad Branch, Islamic Azad University, Mashhad, Iran

\noindent \textit{E-mail address:} hrmoradi@mshdiau.ac.ir}

\end{document}